\documentclass[12pt]{amsart}

\usepackage{amsmath,mathrsfs,amsthm,graphicx,bm,amssymb,enumerate}
\usepackage{txfonts}

\usepackage[makeroom]{cancel}
\usepackage{tikz}
\usepackage[all]{xy} 
\usepackage{enumerate}
\usepackage{lmodern,bm}

\newtheorem{fact}{Fact}[section]   

\newtheorem{theorem}[fact]{Theorem}
\newtheorem{proposition}[fact]{Proposition}
\newtheorem{lemma}[fact]{Lemma}
\newtheorem{corollary}[fact]{Corollary}
\newtheorem{remark}[fact]{Remark}

\newtheorem{example}[fact]{Example}
\newtheorem{assumption}[fact]{Assumption}

\def\Sym{{\rm Sym}}

\def\Ell{\mathcal E\mkern-3.7mu\ell\mkern-2mu\ell}
\def\Ellt{\widetilde{\Ell}}

\def\O{\mathscr O}
\def\Hom{{\rm Hom}}

\def\T{{\bf T}}
\def\t{{\mathfrak t}}
\def\gr{Gro\-then\-dieck }
\def\charr{characteristic }
\def\EE{{\mathcal E}}
\def\vt{\vartheta}
\def\f{f}
\def\om{w}
\def\w{{\underline{w}}}

\def\mfw{\w}
\DeclareMathOperator\EH{Ell}

\DeclareMathOperator\TTT{\mathcal T}

\DeclareMathOperator\C{\mathbb C}
\DeclareMathOperator\PP{\mathbb P}

\DeclareMathOperator\Q{{\mathbb Q}}
\DeclareMathOperator{\GL}{GL}
\DeclareMathOperator{\SL}{SL}

\DeclareMathOperator\Z{{\mathbb Z}}

\DeclareMathOperator\W{\mathbf w}

\DeclareMathOperator\pt{pt}

\DeclareMathOperator{\Ext}{\mathrm{Ext}}
\DeclareMathOperator{\codim}{\mathrm{codim}}
\DeclareMathOperator{\scrExt}{\mathscr{E}\mathrm{xt}}

\def\rhoL{\bar \rho^L}

\def\w{{\underline{w}}}
\def\v{{\underline{v}}}

\title{Elliptic classes of Schubert varieties} 

\author{Shrawan Kumar}
\address{Department of Mathematics, University of North Carolina at Chapel Hill, USA}
\email{shrawan@email.unc.edu}

\author{Rich\'ard Rim\'anyi}
\address{Department of Mathematics, University of North Carolina at Chapel Hill, USA}
\email{rimanyi@email.unc.edu}

\author{Andrzej Weber}
\address{Institute of Mathematics, University of Warsaw, Poland}
\email{aweber@mimuw.edu.pl}

\begin{document}
\begin{abstract} 
We introduce new notions in elliptic Schubert calculus: the (twisted) Borisov-Libgober classes of Schubert varieties in general homogeneous spaces $G/P$. While these classes do not depend on any choice, they depend on a set of new variables. For the definition of our classes we calculate multiplicities of some divisors in Schubert varieties, which were only known for full flag varieties before. Our approach leads to a simple recursions for the elliptic classes. Comparing this recursion with R-matrix recursions of the so-called elliptic weight functions of Rimanyi-Tarasov-Varchenko we prove that weight functions represent elliptic classes of Schubert varieties.

%
%
\end{abstract}

\maketitle

\section{Introduction}

Schubert calculus is usually considered in ordinary cohomology or in $K$-theory. Generalized cohomology theories correspond to formal group laws. Under this correspondence ordinary cohomology and $K$-theory correspond to the one-dimensional algebraic groups $\C$ and $\C^*$ respectively. There is another one-dimensional complex algebraic group, the elliptic curve $E=\C^*/q^{\Z}$, ($|q|<1$ fixed). The corresponding cohomology theory is called {\it elliptic}. In this paper we study the thus obtained (equivariant) {\it elliptic Schubert calculus}. 

A key step in any Schubert calculus is assigning a \charr class to a Schubert variety. Traditionally this \charr class is the {\em fundamental class} notion of the given cohomology theory. However, it is known that in elliptic cohomology the notion of fundamental class is {\it not} well defined \cite{BE}, or in other words, the notion depends on choices. There are important works (e.g. \cite{GR13,LZ} and references therein) on elliptic fundamental classes based on making some natural choices---the choice can be geometric (a resolution) or algebraic (a basis in a Hecke algebra). In this paper we are suggesting a notion which does not depend on choices. Our class is not  the elliptic fundamental class (as just discussed, it does not exist); we regard our class as an analogue of the cohomological {\em Chern-Schwartz-MacPherson (CSM) class}, and the K-theoretic {\em motivic Chern (MC) class}. In fact, certain limits of our elliptic class recovers the CSM and the MC classes. 

The CSM and MC \charr classes are one-parameter deformations of the fundamental classes in their respective cohomology theories. The parameter is usually denoted by\footnote{or by $\hbar$ in physics literature, also sometimes by $y$ in $K$-theory---to match the classical notion of Hirzebruch $\chi_y$--genus} $h$. At ``$h=\infty$'' and $h=1$ the CSM and MC classes specialize to the fundamental class of the theory. Our elliptic class also depends on the extra $h$ parameter. However, the elliptic analogue  has a pole at $h=1$, which we regard as another incarnation of the fact that the notion of fundamental class should not exist in elliptic cohomology. 

Our project--- definition of the $h$-deformed elliptic class of a Schubert variety---has been carried out for full flag varieties $G/B$ in \cite{RW}. Along the way, it was necessary to introduce further new variables\footnote{these extra variables are probably related with the ``dynamical variables'', a.k.a. ``K\"ahler variables'' of mathematical physics literature} $\mu_i$.  The purpose of this paper is to carry out the same task for general homogeneous spaces $G/P$. Compared to the case of $G/B$ some unexpected difficulties need to be handled. The setup of elliptic characteristic classes has a deeply geometric component which is missing from the setup of both CSM classes (in $H_T^*$) and MC classes (in $K_T$). Namely, only special kinds of singularities are allowed (the multiplicities of some divisors of the resolution are constrained) and the pull-back of a Cartier divisor (involving the canonical divisor and the boundary divisor) need to be understood. This piece of geometry was not known for general $G/P$ before. 

Hence, in the first part of the paper we study the divisors and their pullbacks on Schubert varieties of $G/P$. In the second part, using these results, we define the elliptic classes of Schubert varieties in $G/P$ and discuss their defining recursions. In the third part, for $G=\GL_n$, we prove that the thus obtained elliptic class can be represented by an explicit function called  elliptic weight function of \cite{RTV}. 

\medskip
Let us describe some recent developments on the frontiers of geometry and representation theory, which was a guidance of our work, and which may put our construction in context. In a theory initiated by Okounkov and his coauthors \cite{MO, O, AO}  a new \charr class is defined under the name of {\em stable envelope (class)} (see also works of Rim\'anyi-Tarasov-Varchenko \cite{RTV15a, RTV15, RTV}). Stable envelopes have cohomological, K-theoretic, and elliptic versions. Roughly speaking this class is defined as follows: an identification is set up between the Bethe algebra of certain quantum integrable systems and the regular representations of certain cohomology, $K$-theory, elliptic cohomology algebras. On the physics side of this identification there are two natural bases: the spin (or coordinate-) basis and the Bethe (or eigen-) basis. The identification matches the Bethe basis with the fundamental classes of torus fixed points on the geometric side. The geometric classes matching the spin basis are given the name of stable envelope classes. The essence of results in \cite{RV, FR, FRW1, AMSS1, AMSS2} is that, in Schubert calculus settings, the cohomological stable envelopes are the CSM classes of Schubert cells, and the K-theoretic stable envelopes are the MC classes of Schubert cells. Hence, it is natural to predict that there is an elliptic generalization of the CSM and MC class. Moreover, that this notion in Schubert calculus matches the elliptic stable envelopes of \cite{AO, RTV}. Exactly this prediction is fulfilled by the results of \cite{RW} and the present paper. Let us emphasize, that although we used the above mentioned works of Okounkov and others as guidance, our work does not rely on them. 

\medskip

\noindent{\bf Acknowledgements.} S.K. is supported by the NSF grant DMS 1802328, R.R. is supported by a Simons foundation grant. A.W. is supported by the research project of the Polish National Research Center 2016/23/G/ST1/04282 (Beethoven~2, German-Polish joint project).

\section{Notation}
Throughout the paper we will use the following notation.

\begin{itemize}
\item $G$ is any semisimple connected, simply-connected complex linear group with Borel subgroup $B$ and maximal torus $\T$. Its Lie algebra is denoted by $\t=\Hom(\C^*,\T)\otimes\C$. The dual of the Lie algebra $\t^*=\Hom(\T,\C^*)\otimes \C$ contains the lattice of integral weights $\t^*_{\Z}=\Hom(\T,\C^*)$, which are identified with characters.  
We will also need the fractional weights $\t^*_{\Q}=\Hom(\T,\C^*)\otimes \Q$.

\item $P$ is a standard parabolic subgroup with the Levi subgroup $L$ containing $\T$, see \cite[Part II,\S1.8]{J03}.

\item $W_P$ is the Weyl group of $P$, i.e., the Weyl group of $L$.

\item $W^{P}$ denotes the smallest length coset representatives in $W/W_{P}$.

\item We denote the dualizing sheaf of a Cohen-Macaulay Scheme $Y$ by $\omega_{Y}$.

\item $X^{P}_{w}\subset G/P$ is the Schubert variety $\overline{BwP/P}$.

\item $\scrExt$ denotes the sheaf $\Ext$.

\item $\rho\in\t^*_{\Z}$ is the (standard) half sum of positive roots of $G$.

\item $\rho^{L}\in\t^*_{\Z}$ is half the sum of positive roots of $L$.

\item $\mathbb{C}_{\lambda}$ denotes the one dimensional representation of $\T$ as well as the trivial line bundle on $X_{P}:=G/P$ with the $\T$-equivariant structure given by $\mathbb{C}_{\lambda}$.

\item For any character $\lambda$ of $P$, $\mathscr{L}^P(\lambda)$ denotes the line bundle over $X_P$:
$$
G \times^P\, {\mathbb{C}_{-\lambda}}\to X_{P}.
$$

\item 
\begin{tabbing}
Define  $\rhoL$ by $\rhoL(\alpha^{\vee}_{i})$ \== 1, if $\alpha_{i}$ is a simple root of $L$\\[3pt]
                             \>= 0, otherwise.
\end{tabbing}
\end{itemize}
Observe that $\rho - \rhoL$ is a character of $P$ and so is $2\rho - 2\rho^L$. We often identify a character $\lambda$ by its derivative $\dot{\lambda}$. 

\section{The canonical divisor}

The canonical sheaf is the key object of our consideration. Suppose $X$ is a Cohen-Macaulay scheme, then the dualizing complex is concentrated in the gradation: $\dim X$ and coincides up to a shift with the dualizing sheaf $\omega_X$ defined in \cite[\S III.7]{Har}.
Let $j:V\to X$ be the inclusion of an open subset whose complement is of codimension at least 2. By \cite[Lemma 2.7]{KN97} or \cite[\S5]{Ko13} the canonical sheaf is determined by its restriction to $V$:
\begin{equation}\label{pullpush}\omega_X=j_*j^{-1}\omega_X\,.\end{equation}
It is easy to see  that the dualizing sheaf of the homogeneous space is given by
\begin{equation}\label{ambientcanonical}\omega_{X_{P}}=\mathscr{L}^P(-2\rho+2\rho^{L})\,,\end{equation} see, e.g., ~\cite[Part II,\S4.2]{J03}. Moreover,  $X^{P}_{w}$ is a Cohen-Macaulay variety
(\cite [Corollary 3.4.4]{BK}).  Recall that, for any Cohen-Macaulay subvariety $Y$ of a smooth variety $X$,
\begin{equation}\label{subcanonical}
\omega_Y\simeq \scrExt^{\codim Y}_{\O_X}(\O_Y,  \O_X)\otimes \omega_X.\end{equation}
In particular, 
\begin{equation}\label{can_as_ext}
\omega_{X^{P}_{w}}=\scrExt_{\O_{X_{P}}}^{\codim X^{P}_{w}}\left(\O_{X^{P}_{w}},\O_{X_{P}}\right)\otimes \omega_{X_{P}}.
\end{equation}
 
We identify the fixed points of $X^P$ under the action of $\T$ with the set of shortest representatives $W^P\subset W$. For $v,w\in (X^P)^\T$ we  write $v\to w$ if $v<w$ and $\dim X^P_v=\dim X^P_w-1$. Let $\mathring{X}^{P}_{v}\subset X^P_v$ denote the Schubert cell.
For $w\in W^{P}$, let $i_{w}:\{pt\}\to X_{P}$ be 
the map sending the point to the fixed point $w$. Then, as $\T$-equivariant line bundles, $i^{*}_{w}\mathscr{L}^P(\lambda)=\mathbb{C}_{-w\lambda}$, for any character $\lambda:P\to \mathbb{C}^{*}$.
Let $$\xi^{w}:=\mathbb{C}_{\rho-w\bar{\rho}^L}\otimes \omega_{X^{P}_{w}}\otimes \mathscr{L}^P(\rho-\rhoL).
$$

\begin{lemma}\label{lem1}
Restricted to
$\mathring{X}^{P}_{w}$, we have a $B$-equivariant isomorphism:
$$
(\xi^{w})_{|\mathring{X}^{P}_{w}}\simeq (\O_{X^{P}_{w}})_{|\mathring{X}^{P}_{w}}.
$$
\end{lemma}

\begin{proof}
Since $\mathring{X}^P_w$ is smooth isomorphic to an affine space and both sheaves are trivial of rank one,
it is enough  to show that $i^{*}_{w}\xi^{w}$ is trivial as a $\T$-module.
This follows since by \eqref{ambientcanonical} and \eqref{subcanonical}
\begin{align*}
\xi^{w}&=\mathbb{C}_{\rho-w\bar{\rho}^L}\otimes \left(\scrExt^{\codim X^{P}_{w}}_{\O_{X_{P}}}(\O_{X^{P}_{w}},\O_{X_{P}})\otimes \mathscr{L}^P(-2\rho+2\rho^{L})\right)\otimes
\mathscr{L}^P(\rho-\rhoL)\\
&=\C_{\rho-w\bar{\rho}^L}\otimes \scrExt^{\codim X^{P}_{w}}_{\O_{X_{P}}}(\O_{X^{P}_{w}},\O_{X_{P}})\otimes \mathscr{L}^P(-\rho+2\rho^L-\bar{\rho}^L)
\end{align*}
 and 
\begin{align*}
i^{*}_{w}\left(\scrExt^{\codim X^{P}_{w}}_{\O_{X^{P}}}\left(\O_{X^{P}_{w}},\O_{X_{P}}\right)\right)&\simeq \det \left(\frac{T_{w}(X_{P})}{T_{w}(X^{P}_{w})}\right)\\
&\simeq \C_{-(\rho+w\rho-2w\rho^{L})}.
\end{align*}
To prove the last equality we proceed as follows: Let $R^+$ (resp. $R^-$) be the set of positive (resp.~negative) roots of $\mathfrak g$, $R^-_P$ the set of negative roots of the Levi subgroup of $P$. Then $$T_w(X^P_w)=T_w(BwP/P)=T_w(w(w^{-1}Bw\cap B^-)P/P) =\bigoplus_{\beta\in R^+\cap w R^-} \mathfrak g_\beta\,,$$
$$T_w(X_P)=\bigoplus_{\beta\in w(R^-\setminus R^-_P)} \mathfrak g_\beta\,.$$
Thus,
\begin{equation}\label{tangentdet}\det \left(\frac{T_{w}(X_{P})}{T_{w}(X^{P}_{w})}\right)\simeq 
\C_{-w(2\rho-2\rho^L)-(\rho-w\rho)}=\C_{-(\rho+w\rho-2w\rho^{L})}
\end{equation}
by \cite[Cor.~1.3.22(3)]{Ku02}.
The conclusion of the lemma follows since the weight of $i_w^*\xi_w$ is equal to
$$(\rho-w\rhoL)-(\rho+w\rho-2w\rho^L)-w(-\rho+2\rho^L-\bar{\rho}^L)=0\,.$$
\end{proof}
Let $V^{P}_{w}:= \mathring{X}_{w}^{P}\cup \bigcup\limits_{v\to w}\mathring{X}^{P}_{v}$.
Then, $V^{P}_{w}$ is a smooth open subset of $X^{P}_{w}$. 
The restriction $\xi^{w}|_{V^{P}_{w}}$ is an invertible $B$-equivariant $\O_{V^P_w}$-module. Hence, by Lemma \ref{lem1},  
$$
\xi^{w}_{|_{V^{P}_{w}}}\simeq \O_{X^{P}_{w}}\big(-\sum\limits_{v\to w}m^P_{w,v}X^{P}_{v}\big)_{|_{V^{P}_{w}}},\quad \text{for some~ } m^P_{w,v}\in \mathbb{Z}.
$$

\begin{lemma}\label{lem2} The coefficients $m^P_{w,v}$ of the restriction of $\xi^w$ to $V^{P}_{w}$ are given by the formula:
 $$m^P_{w,v}:=1-\langle w\rhoL,\beta^{\vee}\rangle\,,$$ where $\beta$ is the  positive root such that $v=s_{\beta}w$.
\end{lemma}
Here the bracket $\langle-,-\rangle$ denotes the pairing between weights and coweights.

\begin{proof}
Take  $v\in W^{P}$ with $v\to w$. Then,
\begin{align}\label{eq0}
i^{*}_{v}\left(\scrExt^{\codim X^{p}_{w}}_{\O_{X_{P}}}\left(\O_{X^{P}_{w}},\O_{X^{P}}\right)\right) &\simeq \det \left(\frac{T_{v}(X_{P})}{T_{v}(X_{w}^{P})}\right)\notag\\[3pt]
&\simeq \det \left(\frac{T_{v}(X_{P})}{T_{v}(X^{P}_{v})}\right)\otimes \det \left(\frac{T_{v}(X^{P}_{w})}{T_{v}(X^{P}_{v})}\right)^{*}\notag\\[3pt]
&\simeq\C_{-(\rho+v\rho-2v\rho^{L}) +\beta},\quad\text{where}\,\,v=s_{\beta}w ,
\end{align}
by \eqref{tangentdet}. Thus, by \eqref{can_as_ext}, \eqref{ambientcanonical} and \eqref{eq0},
\begin{align}\label{eq1}
i^{*}_{v}\xi^{w} &=\C_{\rho-w\bar{\rho}^L}\otimes i^*_v\left(\scrExt_{\O_{X_{P}}}^{\codim X^{P}_{w}}\left(\O_{X^{P}_{w}},\O_{X_{P}}\right)\otimes \omega_{X_{P}}\right)\otimes i^*_v\mathscr{L}^P(\rho-\rhoL)\notag\\
&=\C_{\rho-w\bar{\rho}^L}\otimes i^*_v\left(\scrExt_{\O_{X_{P}}}^{\codim X^{P}_{w}}\left(\O_{X^{P}_{w}},\O_{X_{P}}\right)\otimes \mathscr{L}^P(-2\rho+2\rho^L)\right)\otimes i^*_v\mathscr{L}^P(\rho-\rhoL)\notag\\
&\simeq \C_{\rho-w\rhoL}\otimes \C_{-(\rho+v\rho-2v\rho^{L}) +\beta}\otimes \C_{v(2\rho-2\rho^L)}\otimes \C_{-v(\rho-\rhoL)}\notag\\
&\simeq\C_{{\beta(1-\langle w\rhoL,\beta^{\vee}\rangle)}},
\end{align}
as the following calculation shows.
\begin{align*}
{\rho-w\rhoL}{-(\rho+v\rho-2v\rho^{L}) +\beta+v(2\rho-2\rho^L)}-v(\rho-\rhoL)
&=\;{v\rhoL+\beta-w\rhoL}\\
&=\;{s_{\beta}w\rhoL+\beta-w\rhoL}\\
&=\;{- \langle w\rhoL,\beta^{\vee}\rangle\beta+\beta}\\
&=\;{\beta(1-\langle w\rhoL,\beta^{\vee}\rangle)}.
\end{align*}
Also,
\begin{align}\label{eq2}
i^{*}_{v}\left(\O_{X^{P}_{w}}\left(-\sum\limits_{u\to w}m^P_{w,u} X^{P}_{u}\right)\right) &= \det \left(\frac{T_{v}(X^{P}_{w})}{T_{v}(X^{P}_{v})}\right)^{\otimes -m^P_{w,v}}\notag\\[3pt]
&= \mathbb{C}_{-\beta}^{\otimes -m^P_{w, v}}\notag\\
&=\mathbb{C}_{m^P_{w, v}\beta}.
\end{align}
Equating \eqref{eq1} and \eqref{eq2}, we obtain the lemma.
\end{proof}

\begin{theorem}\footnote{The proof  is parallel to that of  \cite[$\S$10]{Ku17}}\label{thm3}
For any $w\in W$, we have a $B$-equivariant isomorphism:
$$
{\xi}^{w}\simeq \O_{X^{P}_{w}}\left(-\sum\limits_{v\to w}m^P_{w,v}X^{P}_{v}\right),
$$
where $m^P_{w,v}$ is as in Lemma \ref{lem2}.
Thus, the dualizing  sheaf $\omega_{X^{P}_{w}}$ of $X^{P}_{w}$ is $\T$-equivariantly isomorphic to
$$
\mathbb{C}_{-\rho+w\rhoL}\otimes \O_{X^{P}_{w}}\left(-\sum\limits_{v\to w}m^P_{w,v}X^{P}_{v}\right)\otimes \mathscr{L}^P(\rhoL-\rho).
$$
Moreover, the multiplicity  $m^P_{w,v}$ is a positive integer.
\end{theorem}

\begin{proof} We first  prove the positivity of  $m^P_{w,v}$ :
We have $v<w$, $\beta>0$ and $s_\beta w=v$.
Hence, $v^{-1}<w^{-1}$ and $w^{-1}s_\beta=v^{-1}$.
By \cite[Lemma 1.3.13]{Ku02} the root $w^{-1}\beta$ is negative. Hence, $$m^P_{w,v}=1-\langle w\rhoL,\beta^{\vee}\rangle=1-\langle\rhoL,w^{-1}\beta^{\vee}\rangle\geq1\,.$$

Let $j:V^{P}_{w}\hookrightarrow X^{P}_{w}$ be the inclusion.
Consider the following commutative diagram with exact rows, where $D:= \sum\limits_{v\to w}m^P_{w,v}X^{P}_{v}$ is the divisor with 
$m^P_{w,v}$ as in Lemma \ref{lem2}.
\[
\xymatrix{
0\ar[r] & \O_{X^{P}_{w}}(-D)\ar[d]\ar[r] & \O_{X^{P}_{w}}\ar[r]\ar[d]^{
\simeq} & \O_{D}\ar[r]\ar[d] & 0\\
0\ar[r] & j_{\ast}j^{-1}(\O_{X^{P}_{w}}(-D))\ar[r] & j_{*}j^{-1}(\O_{X^{P}_{w}})\ar[r] & j_{*}j^{-1}\O_{D}. & 
}
\]
The middle vertical arrow is an isomorphism since $X^{P}_{w}$ is normal and $X^{P}_{w}\backslash V^{P}_{w}$ is of $\codim \geq 2$ in $X^{P}_{w}$. Moreover, the right vertical map is injective since the closure of $\bar{D}\cap V^{P}_{w}$ coincides with $\bar{D}$, where $\bar{D}$ denotes the support of $D$. Hence, the left vertical map
\begin{equation}\label{eqnnew0}
\O_{X^{p}_{w}}(-D)\to j_{*}j^{-1}(\O_{X^{P}_{w}}(-D))
\end{equation}
is an isomorphism.

On the other hand, since $\xi^{w}$ is a Cohen-Macaulay $\O_{X^{P}_{w}}$-module, by \eqref{pullpush},  we have
$$
\xi^{w}\,\,\displaystyle{\mathop{\simeq}\limits^{\phi_1}}\,\,j_{*}j^{-1}(\xi^{w})\,\,\displaystyle{\mathop{\simeq}\limits^{\phi_2}}\,\,j_{*}j^{-1}\O_{X^{P}_{w}}(-D)\,\,\displaystyle{\mathop{\simeq}\limits^{\phi_3}}\,\,\O_{X^{P}_{w}}(-D),
$$
where the isomorphism $\phi_2$ follows from Lemma \ref{lem2} and $\phi_3$ is an isomorphism by \eqref{eqnnew0}.
This proves the theorem.
\end{proof}

\begin{corollary} \label{coro2.5} Let $K_{X_\om^P}$ denote a divisor corresponding to the dualizing sheaf $\omega_{X_\om^P}$. Then, 
$$K_{X_\om^P}+ \sum\limits_{v\to w}
m^P_{\om,v} X_v^P$$
is a $\T$-equivariant Cartier divisor representing
$\C_{-\rho+\om\rhoL}
\otimes \mathscr{L}^P(\rhoL-\rho).$ \qed
\end{corollary}

\section{Chevalley formula}

Let $\mfw=s_{j_1}\ldots s_{j_\ell}$ 
 be a reduced decomposition of $w\in W$ ($W$ being the Weyl group of $G$ and $s_j$ are simple reflections) and let $Z_{\mfw}$ be the corresponding Bott-Samelson-Demazure-Hansen resolution$$f_{\mfw}:Z_{\mfw}\to X_{w}^B\subset G/B\,,$$
---which is often called standard resolution, or  Bott-Samelson resolution, or for short BSDH resolution, see, e.g., \cite[$\S$2.2.1]{BK}.

\begin{proposition}\label{prop3.1}
For any integral weight $\lambda\in\t^*_{\Z}$ (not necessarily dominant),  we have
$$
\f^{*}_{\mfw}(\mathscr{L}^B(\lambda))\simeq \O_{Z_{\mfw}}\left(\sum\limits^{\ell}_{i=1}\langle \lambda,\gamma^{\vee}_{i}\rangle \partial_{i}Z_{\mfw}\right),
$$
where
$
\gamma_{i}:= s_{j_\ell}s_{j_{\ell-1}}\ldots s_{j_{i+1}}\alpha_{j_i}$, $\alpha_j$ is the simple root corresponding to the simple reflection $s_j$
and
$$ \partial_{i}Z_{\mfw} :=Z_{s_{j_1}s_{j_2}\ldots \widehat{s_{j_i}}\cdots s_{j_\ell}}.$$
\end{proposition}
\begin{proof}
Consider the diagram
\[
\xymatrix{
Z_{\mfw}\ar[d]_{\pi_{\mfw}}\ar[r]\ar@/^{2pc}/[rr]^-{f_{\mfw}} & X^B_{w[\ell]}\times_{G/P_{\alpha_{j_\ell}}}G/B\ar[r]\ar[d] & G/B\ar[d]\\
Z_{\mfw[\ell]}\ar[r]_{f_{\mfw[\ell]}} & X^B_{w[\ell]}\ar[r] & G/P_{\alpha_{j_\ell}},
}
\]
where $Z_{\mfw[\ell]}$ corresponds to the word $s_{j_1}\ldots s_{j_{\ell-1}}$ and $X^B_{w[\ell]}:= X^B_{s_{j_1}\ldots s_{j_{\ell-1}}}$.
By \cite[\S2, Lemma 3]{Kem},
\begin{align*}
f^{*}_{\mfw}(\mathscr{L}^B(\lambda)) &\simeq \pi^{*}_{\mfw}(f_{\mfw[\ell]}^*(\mathscr{L}^B(s_{j_\ell}\lambda)))\otimes \O_{Z_{\mfw}}\left(\langle \lambda, \alpha^{\vee}_{j_\ell}\rangle Z_{\mfw[\ell]}\right)\\
&=  \O_{Z_{\mfw}}\left(\sum\limits^{\ell - 1}_{i=1}\langle s_{j_\ell}\lambda, s_{j_{\ell-1}}\ldots s_{j_{i+1}}\alpha^{\vee}_{j_i}\rangle(\partial_{i}Z_{\mfw})\right)+
  \O_{Z_{\mfw}}\left(\langle \lambda,\alpha^{\vee}_{j_\ell}\rangle(\partial_{\ell}Z_{\mfw})\right),\,\, \text{by induction on~ } l(w)\\
&=\O_{Z_{\mfw}}\left(\sum\limits^{\ell}_{i=1}\langle \lambda,s_{j_\ell}s_{j_{\ell -1}}\ldots s_{j_{i+1}}\alpha^{\vee}_{j_i}\rangle (\partial_{i}Z_{\mfw}) \right).
\end{align*}
This proves the proposition.
\end{proof}

As a corollary of the above proposition and Corollary \ref{coro2.5}, we get the following.
\begin{corollary}\label{coro2}
$$
\bar{f}^{*}_{\mfw}\left(\omega_{X^{P}_{w}}\otimes \O_{X^{P}_{w}}\left(\sum\limits_{v\to w}m^P_{w,v}X^{P}_{v}\right)\right)\simeq \omega_{Z_{\mfw}}\otimes \O_{Z_{\mfw}}\left(\sum\limits^{\ell}_{i=1}m^P_{\mfw,i}(\partial_{i}Z_{\mfw})\right)\otimes \mathbb{C}_{w\bar{\rho}^L},
$$
where $\pi^P:X_w^B \to X_w^P$ is the projection,  $\bar{f}_{\mfw} := \pi^P\circ f_{\mfw}$ and 
\begin{align*}
m^P_{\mfw,i} : &= 1-\langle w\rhoL,\beta^{\vee}_{i}\rangle,\,\,\,\text{and $\beta_i:= s_{j_1}\ldots s_{j_{i-1}}\alpha_{j_i}$}\\
                 &= 1+\langle \rhoL,\gamma^{\vee}_{i}\rangle  \geq 1.
\end{align*}
\end{corollary}

\begin{proof}
By Corollary \ref{coro2.5},
\begin{equation} \label{eq3.2.1} \omega_{X^{P}_{w}}\otimes \O_{X^{P}_{w}}\left(\sum\limits_{v\to w} m^P_{w,v}X^{P}_{v}\right)\simeq \mathbb{C}_{-\rho+w\rhoL}\otimes \mathscr{L}^P(\rhoL-\rho).
\end{equation}
Further, by \cite [Proposition 2.2.2]{BK},
\begin{equation} \label{eq3.2.3}\omega_{Z_{\mfw}}\simeq \O_{Z_{\mfw}}\left(-\sum\limits^{\ell}_{i=1}\partial_{i}Z_{\mfw}\right)\otimes f^{*}_{\mfw}(\mathscr{L}^B(-\rho))\otimes \mathbb{C}_{-\rho}.
\end{equation} 
From the equation \eqref{eq3.2.1}, we obtain
\begin{align*}
& \bar{f}^{*}_{\mfw}\left(\omega_{X^{P}_{w}}\otimes \O_{X^{P}_{w}}\left(\sum\limits_{v\to w}m^P_{w,v}X^{P}_{v}\right)\right)\simeq \bar{f}^{*}_{\mfw}\left(\mathscr{L}^P(\rhoL-\rho)\right)\otimes \mathbb{C}_{-\rho+w\rhoL}\\[3pt]
&\qquad \simeq f^{*}_{\mfw}(\mathscr{L}^B(-\rho))\otimes f^{*}_{\mfw}\left(\mathscr{L}^B(\rhoL)\right)\otimes \mathbb{C}_{-\rho+w\rhoL}\\[3pt]
&\qquad \simeq  \O_{Z_{\mfw}}\left(\sum\limits^{\ell}_{i=1}\langle \rhoL,\gamma_i^{\vee}\rangle(\partial_{i}Z_{\mfw})\right)\otimes f^{*}_{\mfw}(\mathscr{L}^B(-\rho))\otimes \mathbb{C}_{-\rho+w\rhoL},\,\,\text{by Proposition \ref{prop3.1}}\\[3pt]
&\qquad = \O_{Z_{\mfw}}\left(\sum\limits^{\ell}_{i=1}\left(m^P_{\mfw,i}-1\right)(\partial_{i}Z_{\mfw})\right)\otimes f^{*}_{\mfw}(\mathscr{L}^B(-\rho))\otimes \mathbb{C}_{-\rho+w\rhoL}\\[3pt]
&\qquad = \O_{Z_{\mfw}}\left(\sum\limits^{\ell}_{i=1}m^P_{\mfw,i}(\partial_{i}Z_{\mfw})\right)\otimes \omega_{Z_{\mfw}}\otimes \mathbb{C}_{w\rhoL},\quad \text{by  \eqref{eq3.2.3}.}
\end{align*}
This proves the corollary.
\end{proof}
\begin{remark}\label{remark4.3}\rm 
In the case of $P=B$ the Corollary \ref{coro2.5} specializes to \cite[Theorem 3.2]{RW} or \cite[Exercise 3.4.E.1]{BK}. In this case $m^B_{w, v}=1$ for any
$v\to w$. 
Moreover, in this case,
 $\sum_{i=1}^{\ell}m^B_{\mfw,i}\,(\partial_i Z_\mfw)=\partial Z_\mfw$ since each $m^B_{\mfw,i} = 0$ by definition, where $\partial Z_\mfw :=\sum_{i=1}^\ell \partial_i Z_\mfw$.  .
\end{remark}

\section{The Borisov-Libgober elliptic characteristic class} \label{sec:BLgeneral}

We will study the Borisov-Libgober elliptic characteristic class of certain pairs $(X,\Delta)$. It is defined in \cite{BoLi0, BoLi1,BoLi2}, and the version we consider is in \cite{RW}. Here we recall the main definitions in a special case (torus-equivariant case with finitely many fixed points), which is sufficient for the purpose of this paper.

\subsection{Smooth case}
First, let $Z$ be a smooth variety and $D$ a simple normal crossing divisor. Assume that a torus $\T$ acts on $Z$ leaving $D$ stable. One can consider the elliptic class of $(Z,D)$ either in the $\T$-equivariant $K$-theory, or in the $\T$-equivariant elliptic cohomology%
\[
\Ellt(Z,D)\in K_{\T}(Z)(h),\qquad \Ellt^E(Z,D)\in \EH_{\T}(Z)(h).
\]
{Here we use elliptic cohomology in its traditional sense: it is a generalized complex-oriented cohomology theory, see \cite{La}. Because of a lack of a convenient definition of {\it equivariant} elliptic cohomology we rather study the image of the 
elliptic class in Borel equivariant cohomology or $K$-theory, see \cite[\S3]{MW}.
In the sense of recent approaches to $\EH_{\T}$, as in  \cite[Section 2]{AO}, \cite[Section 4]{FRV18}, or \cite[Section 7]{RTV}, elements of our $\EH_{\T}$ are sections of certain line bundles over the {\em elliptic cohomology scheme} considered in those works. The Euler class of a vector bundle is a section of a {\it Thom bundle}, see \cite[\S7]{Ga}.}

Here is the  definition of the elliptic class in the special case when $|Z^{\T}|<\infty$. In this case they are defined by their restrictions to $\T$-fixed points. For a fixed point $x$ we have
\begin{align*}
\Ellt(Z,D)_{|x}=&e(T_xZ) \prod_{k=1}^{\dim Z} \frac{\vt( \chi_k\, h^{1-a_k})\vt'(1)}{\vt(\chi_k)\vt(h^{1-a_k})},
\\ \\\Ellt^E(Z,D)_{|x}=&e^E(T_xZ) \prod_{k=1}^{\dim Z} \frac{\vt( \chi_k\, h^{1-a_k})\vt'(1)}{\vt(\chi_k)\vt(h^{1-a_k})}=\prod_{k=1}^{\dim Z} \frac{\vt( \chi_k\, h^{1-a_k})\vt'(1)}{\vt(h^{1-a_k})},
\end{align*}
where the products are taken with respect to the  equivariant coordinates at $x$ 
and
\begin{itemize}
\item $\chi_k\in K_\T(\{x\})=R(\T)$ is the character of the $k$-th coordinate;
\item $e(T_xZ)=\prod(1-\chi_k^{-1})$ and $e^E(T_xZ)=\prod \vt(\chi_k)$ are the equivariant Euler classes in $K$-theory and elliptic cohomology;
\item $a_k\in\Q$ is the multiplicity of the divisor along the $k$-th coordinate; and
\item $$\vt(x)=
(x^{1/2}-x^{-1/2})\prod_{n\geq 1}(1-q^n x)(1-q^n /x)
$$ is (a version of) the theta function which is considered in \cite{RTV, RW}. 
Here it is treated as a formal series in $x^{\pm1/2}$. The variable $q$ is treated as a constant. 
\end{itemize}
We have to assume  that  1 does not appear among the multiplicities of $D$, otherwise we have 0 in the denominator.

It is worth getting rid of the dependence on which cohomology theory we are in, and work with the elliptic class
\[
\EE(Z,D)=\frac{\Ellt(Z,D)}{e(TZ)} = \frac{\Ellt^E(Z,D)}{e^E(TZ)}.
\]
Then, using the notation 
\[
\delta(x,y)=\frac{\vt(x y)\vt'(1)}{\vt(x)\vt(y)}
\]
we have
\[
\EE(Z,D)_x:=\EE(Z,D)_{|x}=\prod_{k=1}^{\dim Z} \delta( \chi_k, h^{1-a_k}).
\]

\subsection{Singular case}
Let the $\T$-stable singular pair $(X,\Delta)$ be embedded in a smooth ambient $\T$-variety $M$, and assume that $K_X+\Delta$ is $\Q$-Cartier. The $\T$-equivariant elliptic classes of the pair $(X,\Delta)$ are defined by
\[
\Ellt(X,\Delta;M)=f_*\Ellt(Z,D)\in K_{\T}(M)(h), \qquad \Ellt^E(X,\Delta;M)=f_*\Ellt^E(Z,D)\in \EH_{\T}(M)(h)
\]
where $f:Z\to X$ is a $\T$-equivariant resolution of singularities and $K_Z+D=f^*(K_X+\Delta)$. 
If the multiplicities of $D$ are smaller\footnote{The discrepancy divisor is equal to $-D$. We  do not assume that $\Delta$ is effective.
} than 1, then the definition does not depend on the resolution, by \cite{BoLi1}.

\begin{assumption}\label{assu}To have well defined elliptic class we assume that the coefficients of $D$ are smaller than 1.\end{assumption}

Just like in the smooth case, it is worth considering the version
\begin{equation}\label{EEdef}
\EE(X,\Delta)=\frac{\Ellt(X,\Delta;M)}{e(TM)}=\frac{\Ellt^E(X,\Delta;M)}{e^E(TM)}\;\in\; e(TM)^{-1}K_{\T}(X)(q,h) .
\end{equation}
Note that, assuming $|M^{\T}|<\infty$, the Euler 
class $e(TM)$ is not a zero divisor in the 
localization $S^{-1}K_{\T}(M)$, where $S\subset 
K_{\T}(pt)={\rm R}(\T)$ is the multiplicative system generated by $1-\C_\lambda$, $\lambda\in\t^*_{\Z}$.
Assuming that the number of torus fixed points on $X$ and $Z$ are finite, the restriction of $\EE(X,\Delta)$ to a $\T$-fixed point $x$ will be denoted by $\EE(X,\Delta)_x$. These latter classes are elements of the fraction field of $K_{\T}(\pt)(q,h)$, and are also independent of the ambient manifold $M$---that is why we dropped $M$ from the notation.

\subsection{Push-forward}\label{sec:push}
In the  case we study, i.e.,  that of finitely many $\T$-fixed points, the push-forward map $f_*$ can be described as follows. Let $(Z,D)$ be the resolution of $(X,\Delta)$ as above and $x$ a $\T$-fixed point in $X$. Then, according to Lefschetz-Riemann-Roch, which is the equivariant localization description of push-forward maps \cite[Thm. 5.11.7]{ChGi},  
we have 
\[
\EE(X,\Delta)_x=\sum_{y\in f^{-1}(x) \cap Z^{\T}} \EE(Z,D)_y.
\]

\section{Elliptic classes of Schubert varieties}
Our main object of study is the  equivariant elliptic characteristic classes of Schubert varieties, living in the $\T$-equivariant $K$-theory or elliptic cohomology of $G/P$. By the nature of the definition of elliptic classes (see Section~\ref{sec:BLgeneral}) we need to consider not the Schubert varieties or the Schubert cells themselves, but pairs $(X^P_w,\Delta^P_\om)$, where $\Delta^P_\om$ is a certain $\T$-stable $\Q$-divisor contained in $\partial X^P_\om$, such that $K_{X^P_w}+\Delta^P_w$ is $\Q$-Cartier. 

\subsection{The class $E(X^P_\om)$}\label{ourdivisor}

{Let $X^P_\om$ be a Schubert variety in $G/P$,  $\lambda$ a character of $P$ and assume that the line bundle $\mathscr{L}^P(\lambda)$ over $X_w^P$  is ample. Let $\Delta^P_{\om,\lambda}$ be the zero divisor of the unique (up to scalar multiples)  $U$-invariant section (eigenvector) of $ \mathscr{L}^P(\lambda)_{|X^P_\om}$, where $U:=[B,B]$ is the unipotent radical of $B$. Then, the support of $\Delta^P_{\om,\lambda}$ is precisely equal to $\partial X^P_w :=\cup_{v\to w}\, X^P_v$. 
 Consider  the pair 
\begin{equation}\label{eqn:pair}
(\;X_\om^P\;,\;\sum_{v\to w} m^P_{\om,v}\, X^P_v - t \Delta^P_{\om,\lambda}\;),
\end{equation}
where the coefficients $m^P_{\om,v}$ are from Lemma \ref{lem2}. Our main object of study is the $\EE$ class (see~\eqref{EEdef}) of this pair. For this to make sense we need to show that the requirements of such a pair are satisfied. 

\begin{remark}\rm For the case $P=B$ Ganter and Ram \cite{GR13} suggested to consider the boundary divisor equal to $\partial X^B_w-t\Delta^P_{\om,\rho}$ for $0<t\ll 1$, but in our approach instead of $\rho$ we allow any weight $\lambda$ defining a sufficiently ample line bundle on $X^P_w$.
\end{remark}

Fix a reduced word $\w$ of $w$. As earlier, let $\bar f_\w:Z_\w\to X_\om^P$ be the composition of the BSDH resolution $f_\mfw: Z_\mfw \to X_w^B$ with the quotient map $X_w^B\to X_w^P$. Let
$$\bar f^*_\w\left(K_{X_\om^P}+\sum_{v\to w}\, m^P_{\om,v}\, X_v^P - t\Delta^P_{\om,\lambda}\right)=K_{Z_\w}+\sum_{i=1}^\ell a_i \partial_i Z_\w\,.$$
By  Corollary \ref{coro2} the coefficients $a_i<1$ if $t\gg0$, that is, for large $t$ Assumption~\ref{assu} is satisfied; and the $\EE$ class of \eqref{eqn:pair} is indeed well defined.

Let us rephrase this construction without mentioning $t$: allowing rational weights $\lambda$, the class 
\begin{equation}\label{Epair}
\EE(X_\om^P,\sum_{v\to w} m^P_{\om,v} X^P_v - \Delta^P_{\om,\lambda})
\end{equation}
is well defined for $\lambda$ belonging to a certain open subset of $(\t_{\Q}^*)^{W_P}$. For these $\lambda$'s the dependence of \eqref{Epair} on $\lambda$ is an explicit meromorphic function in $\lambda$ (this follows from the push-forward formalism described in Section \ref{sec:push}). This meromorphic function, now considered for all $\lambda\in (\t^*)^{W_P}$, is our main object: the {\em elliptic class of the Schubert variety $X^P_\om$}. We will denote it by $E(X^P_\om)$, or by $E(X^P_\om,\lambda)$ if we want to emphasize the $\lambda$-dependence. In some calculation below we will assume that ``$\lambda$ is large enough'' so that $E(X^P_\om)$ equals \eqref{Epair};  thus obtained formulas then must hold for the meromorphic function $E(X^P_\om)$.

\subsection{Elliptic classes in the BSDH resolution}
Observe that, by Proposition \ref{prop3.1}, Corollary~\ref{coro2} and Remark \ref{remark4.3}, 
\begin{align}\label{eq:PandB}\bar{f}_\w^*(K_{X^P_w}+\sum_{v\to w} m^P_{\om,v}\, X_v^P - \Delta^P_{\om,\lambda})&=K_{Z_\w}+\partial Z_\w- f_\w^*(\mathscr{L}^B(\lambda-\rhoL))\notag\\
&=f_\w^*(K_{X^B_w}+\partial X^B_w - \Delta^B_{\om,\lambda-\rhoL})
\,,\end{align}
where $\partial Z_\w:= \sum_{i=1}^\ell \partial_i Z_\w$. 
Note that the bundle $ \mathscr{L}^B(\rhoL)$ does not come from $X_w^P$.
The class $E(X^P_\om,\lambda)$ is obtained from 
\[E(Z_\w,\lambda-\rhoL):=\EE(Z_\w\;,\;\partial Z_\w -  f_\w^*(\Delta^B_{w,\lambda-\rhoL}))
\]
using the localization formula in Section \ref{sec:push}.

The class $E(Z_\w,\lambda-\rhoL)$ is determined by its restrictions to the torus fixed points. The $\T$-fixed points of $Z_\w$ are indexed by the subwords of $\v\subset\w$.
From the well-known combinatorial description of the BSDH resolution (see \cite[\S\S3.2 and 4.3]{RW}), we obtain 
\begin{proposition}\label{prop6.1} Let $w\in W$ and let $\w$ be a reduced word for $w$. Then, for any (not necessarily reduced)
subword $\v$ of $\w$, we have:
\begin{equation}\label{eq:EZ}
E(Z_\w,\lambda)_\v=\prod_{i=1}^\ell \delta(e^{-v_{[1,i]}\alpha_{j_i}},\psi(i)),
\end{equation}
where
$v_{[1,i]}$ is the product of $s_{j_k}$'s with $k\leq i$ appearing in $\v$ and
\begin{align*}\psi(i)&=\begin{cases}h&\text{if } \text{$i$-th letter of }\w\text{ is not omitted in }\v\\
h^{\langle\,\lambda\,,\, \gamma_i^\vee\,\rangle}&\text{otherwise\,,}\end{cases}
\end{align*}
where $\gamma_i$ is as in Proposition \ref{prop3.1}.
\end{proposition}
\begin{proof}
The multiplicity of the divisor $f_\w^*(\Delta^B_{w,\lambda})$ along  $\partial_i Z_\w$ is equal to (using Proposition \ref{prop3.1})
$$\langle\,\lambda\,,\, \gamma_i^\vee\,\rangle\,.$$
The tangent weights are the same as in \cite[\S 3.2]{RW}.\end{proof}

Recall that, by Subsection \ref{sec:push}, we have the following: for any $v, w\in W$ (choosing a reduced word $\w$  for $w$):
\begin{equation} \label{eqn6.3.2} E(X_w^B, \lambda)_v= \sum_{\v}\, E(Z_\w, \lambda)_\v,
\end{equation}
where the summation runs over those (not necessarily reduced) subwords $\v$ of $\w$ for which $\mu(\v) = v$. Here 
$\mu(\v)= s_{i_1}\dots s_{i_p}$ for the word $\v= (s_{i_1}, \dots ,  s_{i_p})$.  In particular,
\begin{equation} \label{eqn6.3.3} E(X^B_w, \lambda)_v = 0,\,\,\,\text{if $v\nleq w$}.
\end{equation}

\noindent We note that \eqref{eq:PandB} also implies the following corollary.  

\begin{corollary}\label{cor:Esums}
Let $\lambda\in\t^*$ be a $W_P$-invariant weight, $w\in W^P$. Then, for any $v\in W^P$,
$$E(X^P_w,\lambda)_{v}=\sum_{u\in W_P}E(X^B_w,\lambda-\rhoL)_{vu}\,.$$
(In particular, if $v \nleq w$, then $E(X^P_w,\lambda)_{v}= 0$ by using the above identity and the Identity \eqref{eqn6.3.3}.)

Equivalently,
$$E(X^P_w,\lambda)=\pi^{P}_*E(X^B_w,\lambda-\rhoL)\,,$$
where $\pi^{P}:G/B\to G/P$ is the natural quotient map. \qed
\end{corollary}
In Section \ref{sec:weight}, for $G=\SL_n$ we will identify the functions $E(X_w^P,\lambda)_v$ with some substitutions of well-known special functions called weight functions. Corollary \ref{cor:Esums} seems to be a  new result for those substitutions of weight functions.


\subsection{Recursions}
The Schubert varieties in $G/P$ are parametrized by cosets in $W/W_P$. Our goal is to describe the behavior of the elliptic class when we pass from $w$ to $s_\alpha w$ for a  simple reflection $s_\alpha$ such that  $\dim X^P_{s_\alpha w}>\dim X^P_{w}$. First, we solve the recursion for the BSDH-variety $Z_\w$, which is a resolution of $X_w^B$ as well as  $X_w^P$, provided  $w\in W^P$.
Having an explicit formula for $E(Z_\w,\lambda - \rhoL )_\v$, we obtain a recursion for the classes of the BSDH resolution and then we push  it down to $X^P_w$. It turns out that the recursion is well defined for the elliptic classes of Schubert varieties.

\begin{theorem}\label{thm:Erecursion}
Let $\alpha$ be a simple root and $w\in W^P$. If $\dim X^P_\om<\dim X^P_{s_\alpha \om}$ (in particular, $s_\alpha \om\in W^P$),  then, 
for any coset $[v]\in W/W_P$,
\begin{equation}\label{eqn6.3.1} E(X^P_{s_\alpha\om},\lambda)_{[v]}=
\delta\!\left(e^{-\alpha},h^{\langle\lambda-\rhoL,\om^{-1}\alpha^\vee\rangle}\right) 
\cdot E(X^P_\om,\lambda)_{[v]} +
\delta\!\left(e^{\alpha},h\right)\cdot 
s_\alpha^z\left[E(X^P_\om,\lambda)_{[s_\alpha v]}\right],
\end{equation}
where $s_\alpha^z[-]$ is the action of $s_\alpha$ on the equivariant parameters of $K$-theory.
\end{theorem}
\noindent The notation ``$s_\alpha^z$'' will be justified below. 

\begin{proof} We first define, for any two cosets, $[u]\leq [v]$ if and only if $u'\leq v'$, where $u'$ is the smallest length coset representative in $ [u]$.   If $[v]\nleq [s_\alpha w]$, then so is  $[s_\alpha v]\nleq [w]$ (use \cite[Corollary 1.3.19]{Ku02}). Thus, both the sides of the equation \eqref{eqn6.3.1} are zero by Corollary \ref{cor:Esums}. 

So, assume that $[v]\leq [s_\alpha w]$. Let $s_\alpha \w$ be a reduced word for  $s_\alpha w$. Let $\v$ be a subword (not necessarily reduced) of  $s_\alpha \w$. 
Let $\v'=\v\cap \w$, i.e.,  $\v'=\v$ if the first letter of $s_\alpha \w$ is omitted in $\v$ and $s_\alpha\v'=\v$ otherwise. 
 Then, by Proposition \ref{prop6.1}, 
\begin{equation}\label{eq:EZrecursion}E(Z_{s_\alpha\w},\lambda-\rhoL)_\v=\begin{cases}
\delta\!\left(e^{\alpha},h\right)\, 
s_\alpha^z\cdot [E(Z_\w,\lambda-\rhoL)_{\v'}]\quad\text{if the first letter of $s_\alpha\w$ is not omitted in }\v,\\
\delta\!\left(e^{-\alpha},h^{\langle\lambda-\rhoL,\om^{-1}\alpha^\vee\rangle}\right)\, E(Z_\w,\lambda-\rhoL)_{\v'}\quad\text{otherwise}\,.\end{cases}\end{equation}

To compute $E(X^P_{s_\alpha\om},\lambda)_{[v]}$, by Corollary \ref{cor:Esums} and the Identity \eqref{eqn6.3.2}, we sum up the contributions coming from  $E(Z_{s_\alpha\om},\lambda-\rhoL)_\v$, where $\v$ varies over those  (not necessarily reduced) subwords $\v$ of $s_\alpha\w$ such that $\mu(\v)\in [v]$.
Let us examine the first factor of the product \eqref{eq:EZ}, appearing in \eqref{eq:EZrecursion}:
\begin{itemize}
\item If the first letter of $s_\alpha\w$ is not omitted in $\v$, then the corresponding factor is equal to  
$$\delta\!\left(e^{\alpha},h\right).$$
In the remaining factors the variables in the first argument of $\delta$ should be changed by the action of $s_\alpha$.
\item If the first letter of $s_\alpha\w$ is omitted in $\v$, then the corresponding factor is equal to  
$$
\delta\!\left(e^{-\alpha},h^{\langle\lambda-\rhoL,\om^{-1}\alpha^\vee\rangle}\right)
\,.$$
The remaining factors are unchanged.
\end{itemize}
Therefore,  we obtain two kinds of summands in the decomposition of $E(X^P_{s_\alpha w},\lambda)_{[v]}$, one coming from the subwords $\v$ of $s_\alpha\w$ which do not contain the first letter $s_\alpha$ which contribute to 
$E(X_w^P, \lambda)_{[v]}$ and the other coming from those subwords $\v$ which do contain the  first letter $s_\alpha$ and hence  contribute to 
$E(X_w^P, \lambda)_{[s_\alpha v]}$. 
\end{proof}
The following lemma for $v\neq 1$ follows from Corollary \ref{cor:Esums}, and for $v=1$ it follows easily from the definition.
\begin{lemma}\label{thm:Einitial} For $v\in W^P$, 
we have
\[
E(X^P_{1},\lambda)_v=
\begin{cases}
1 & \text{if } v=1 \\
0 & \text{otherwise.}
\end{cases}
\]
\end{lemma}

The recursion with initial condition presented in Theorem \ref{thm:Erecursion} and Lemma \ref{thm:Einitial} is an effective way of computing the fixed point restrictions of the elliptic classes $E(X^P_\om)$. We invite the reader to verify the initial condition and the recursion in the following examples. In these examples we  consider homogeneous spaces for $G=\SL_n$. It is convenient to extend the action to $\GL_n$ and to have $n$-dimensional maximal torus. We use the notation $z_i=e^{\varepsilon_i^*}$ ($\varepsilon_i$ is the standard basis of $\t=\C^n$ 
for $n=2,3,4$), 
for more general notation for the natural variables of $E(X_w)_v$ for $G=\GL_n$ see the next section.

\begin{example} \rm \label{ex:GL2}
For $G=\SL_2$, $P=B$,  $W=\{1, s\}$, we have:

\centerline{
\begin{tabular}{ll}
$E(X^P_{1})_{1}= 1,$ & $E(X^P_{1})_s=0,$ \\
$E(X^P_{s})_{1}=\delta(z_2/z_1,\mu_2/\mu_1),$ & $E(X^P_s)_s=\delta(z_1/z_2,h)\,,$
\end{tabular}
}
\noindent where $\lambda = (\lambda_1, \lambda_2)$ and $\mu_i=h^{-\lambda_i}$. Observe the obvious triangularity property $v\not\leq \om \Rightarrow E(X^P_\om)_v=0$, and that the `diagonal' restrictions $E(X^P_\om)_\om=\prod_\chi\delta(\chi,h)$ for the weights $\chi$ of $T_\om X^P_\om$. The off-diagonal restrictions may be complicated formulas in general.
\end{example}

\begin{example} \rm \label{ex:GL3}
For $G=\SL_3$, $P=B$, with analogous notation, we have the following fixed point restrictions.
\medskip

{\def\arraystretch{1.5}\begin{tabular}{|c|ccccc}
\hline
& $v=$123 & $v=$132 & $v=$213 & $v=$231 & \ldots \\
\hline
$E(X_{123}^P)_v$ & $1$ & 0 & 0 & 0 & \ldots \\
$E(X_{132}^P)_v$ & $\delta(\frac{z_3}{z_2},\frac{\mu_3}{\mu_2})$ & $\delta(\frac{z_2}{z_3},h)$ & 0 & 0 & \ldots \\
$E(X_{213}^P)_v$ & $\delta(\frac{z_2}{z_1},\frac{\mu_2}{\mu_1})$ & 0 &  $\delta(\frac{z_1}{z_2},h)$ & 0 & \ldots  \\
$E(X_{231}^P)_v$ & $\delta(\frac{z_2}{z_1},\frac{\mu_3}{\mu_1})\delta(\frac{z_3}{z_2},\frac{\mu_3}{\mu_2})$ &  $\delta(\frac{z_2}{z_1},\frac{\mu_3}{\mu_1})\delta(\frac{z_2}{z_3},h)$ 
& $\delta(\frac{z_3}{z_1},\frac{\mu_3}{\mu_2})\delta(\frac{z_1}{z_2},h)$ & $\delta(\frac{z_1}{z_2},h)\delta(\frac{z_1}{z_3},h)$  & \ldots  \\
$\vdots$  & $\vdots$  & $\vdots$  & $\vdots$  & $\vdots$  & $\ddots$ \\
\end{tabular}}
\end{example}


\begin{example} \rm \label{ex:GL3}
Let $G=\SL_4$ and $G/P$ be the Grassmannian of planes in $\C^4$. The cells, as well as the fixed points are indexed by the two-element subsets of $\{1,2,3,4\}$. 
The Weyl group $W_P\subset W=S_4$ is spanned by two transpositions $s_1$ and $s_3$. Let $\mu=\mu(\lambda)=h^{\langle\lambda,\varepsilon_2-\varepsilon_3\rangle}$ be the function in $\lambda\in(\t^*)^{W_P}\subset (\C^4)^*$ given by the exponent of the product with the dual root $\alpha_2^\vee=\varepsilon_2-\varepsilon_3$. 
The  restrictions of $E(X_w^P)$ are presented in the following table:
\medskip

{\small
$\def\arraystretch{1.5}\begin{array}{|c|ccccc}\hline&v=12&v=13&v=14&v=23&\!\!\dots\\
\hline
\!\!E(X_{12}^P)_v\!\!
&1&0&0&0&\!\!\\
\!\!E(X_{13}^P)_v\!\!
&\delta(\frac{z_3}{z_2},{\mu})&\delta(\frac{z_2}{z_3},h)&0&0&\!\!\dots\\
\!\!E(X_{14}^P)_v\!\!
&\!\!\delta(\frac{z_3}{z_4},h)\delta(\frac{z_4}{z_2},{\mu})+\delta(\frac{z_3}{z_2},{\mu})\delta(\frac{z_4}{z_3},\frac{\mu}{h})&\delta(\frac{z_2}{z_3},h)\delta(\frac{z_4}{z_3},\frac{\mu}{h})&\delta(\frac{z_2}{z_4},h)\delta(\frac{z_3}{z_4},h)&0&\!\!\dots\\
\!\!E(X_{23}^P)_v\!\!
&\!\!\delta(\frac{z_1}{z_2},h)\delta(\frac{z_3}{z_1},{\mu})+\delta(\frac{z_2}{z_1},\frac{\mu}{h})\delta(\frac{z_3}{z_2},{\mu})&\delta(\frac{z_2}{z_1},\frac{\mu}{h})\delta(\frac{z_2}{z_3},h)&0&\delta(\frac{z_1}{z_2},h)\delta(\frac{z_1}{z_3},h)&\!\!\dots\\
\vdots&\vdots&\vdots&\vdots&\vdots&\!\!\ddots\\\end{array}$
}

\noindent Comparing with the previous example, here the $\mu$-variable and $h$ may appear together
in one argument of $\delta$. This is due to the presence of the component $\rhoL$, which for $P=B$ vanishes.
\end{example}

\begin{example} \label{ex:Sp} \rm
Let $G=Sp_2$ and let $G/P=LG(2)$ be the Lagrangian Grassmannian of $2$-planes in $\C^4$. It is isomorphic to the quadric in $\PP^4$.
The Weyl group $W$ of $G$ is generated by the transposition $s_1=\begin{pmatrix}0&1\\1&0\end{pmatrix}$ and the sign change $s_2=\begin{pmatrix}1&0\\0&-1\end{pmatrix}$ under the $\{\varepsilon_1, \varepsilon_2\}$ basis as in \cite[Planche III]{Bou}.
The corresponding roots are 
$$\alpha_1=(1,-1),\qquad\alpha_2=(0,2).$$ 
With respect to the standard scalar product the coroots are the following
$$\alpha^\vee_1=(1,-1),\qquad\alpha^\vee_2=(0,1).$$
The weight $\rhoL$ appearing in our computation is equal to $(1,0)$.
The group $W_P$ is generated by $s_1$. 
There are four cells in $LG(2)$ corresponding to the words
$$1,\quad s_2,\quad s_1s_2,\quad s_2s_1s_2.$$ 
The weights, which are invariant with respect to $W_P$ are of the form $(\lambda,\lambda)$.
Let $\mu=h^{\lambda}$.
The elliptic class of the top dimensional Schubert variety restricted to $1$ is equal to
\begin{multline}\label{eqn:ESp}
E(X^P_{s_2s_1s_2})_{1}=\delta \left(\frac{1}{z_2^2},\frac{\mu}{h
   }\right)\delta
   \left(\frac{z_2}{z_1},\frac{ \mu
   ^2}h\right)\delta
   \left(\frac{1}{z_2^2},{\mu
   }\right)+
\\+
   \delta \left(z_2^2,h\right)\delta
   \left(\frac{1}{z_1 z_2},\frac{\mu
   ^2}h\right)\delta
   \left(\frac{1}{z_2^2},h\right)+\delta \left(\frac{1}{z_2^2},\frac{\mu}{h 
   }\right)\delta \left(\frac{z_1}{z_2},h\right)\delta
   \left(\frac{1}{z_1^2},{\mu }\right).\end{multline}
The first summand corresponds to the empty subword, the second one to $s_2s_2$, the third one to~$s_1$. See also Section \ref{sec:trans}.
\end{example}

\section{Elliptic classes of Schubert varieties in type $A$}
\subsection{Notation in type A} \label{sec:GLn_notations}
Let $G=\SL_n$. For convenience we consider the full group of linear transformations $\GL_n$ and  the maximal torus therein. Denote the standard basis of $\t=\C^n$ by $\varepsilon_i$. The simple roots, following standard convention as in \cite[Planche I]{Bou}, are  $\alpha_i=\varepsilon_i-\varepsilon_{i+1} \,(1\leq i \leq n-1)$. Fixing the standard scalar product in $\t_{\Q}=\Q^n$, we identify coroots with roots, that is $\alpha_i=\alpha_i^\vee$. The Weyl group $W=S_n$ is
  identified with the group of permutations of the set $\{1,2,\dots,n\}$, and also with the group of $n\times n$ permutation matrices. Let $\{s_i\}_{1\leq i\leq n-1}\subset W$ be the set of (simple) reflections corresponding to  the simple  roots $\alpha_i$. 

Consider the parabolic subgroup $P$ corresponding to the sequence of positive integers ${\bf k}= (k_1,k_2,\dots,k_m)$ with $\sum k_i=n$. The variety $G/P$ is the partial flag variety parametrizing flags of subspaces $(V_i)_{i=1,\dots,m}$ with $\dim V_i/V_{i-1}=k_i$. The Weyl group of the Levi factor is 
\[
W_P=S_{k_1}\times S_{k_2}\times\dots\times S_{k_m}\subset W=S_n.
\]
Set $k^{(s)}=\sum_{i=1}^s k_i
$ and $k^{(0)}=0$.
The simple roots of the Levi factor are those 
$
\alpha_i$ such that $\{i,i+1\}\subset[k^{(s-1)}+1, k^{(s)}]$ for some $s\in[1,m].
$
 The defining condition 
$$\langle \rhoL,\alpha_i^\vee\rangle=\begin{cases}1 &\text{if } s_i\in W_P\\
0 &\text{if } s_i\not\in W_P\end{cases}$$
of $\rhoL$ translates to the formula
$\rhoL=(r_1,r_2,\dots,r_n),$ where, for $1\leq i\leq n-1$, 
$$r_i-r_{i+1}=\begin{cases}1&\text{if }\exists s\in[1,\ m]\;\;\text{with}\,\,\{i,i+1\}\subset[k^{(s-1)}+1, k^{(s)}]\\
0&\text{otherwise}.\end{cases}$$
The weight $\rhoL$ is determined by this condition up to the addition of $\mathbb{Z}\Delta$, where $\Delta := (1, \dots, 1)$. For example, if $(k_1,k_2,k_3)=(2,3,1)$, then $\rhoL=(4,3,3,2,1,1)$ or equally well we can take
$\rhoL=(3,2,2,1,0,0)$.

The number $\langle \om\rhoL,\alpha_i^\vee\rangle$ is crucial for our computations. In type $A$ it is rewritten as 
\begin{equation}\label{eq:crucial}
\langle \rhoL,\om^{-1}\alpha_i^\vee\rangle=r_{\om^{-1}(i)}-r_{\om^{-1}(i+1)}.
\end{equation}

\subsection{The recursion for $E(X^P_\om,\lambda)$ in type $A$}
Let us denote the basis characters $\T=(\C^*)^n\to \C^*$ by $z_i=e^{\varepsilon_i^*}$. The exponential of the simple root $\alpha_i$ is hence 
$$e^{\alpha_i} =\frac{z_i}{z_{i+1}}\qquad\text{for }i=1,2,\ldots,n-1.$$
Let $y_i=h^{-\varepsilon_i}$. It is treated as a function on $\t^*=\C^n	$: for $\lambda=(\lambda_1,\lambda_2,\dots,\lambda_n),$
\begin{equation}\label{defy}
y_i(\lambda)= h^{-\langle\lambda,\varepsilon_i\rangle}=h^{-\lambda_i}.
\end{equation}
With this notation, and using \eqref{eq:crucial}, we obtain that the recursion of Theorem \ref{thm:Erecursion} in type $A$ takes the following form for $v, w\in W^P$ and a simple reflection $s_i$ such that $s_iw\in W^P$ and $s_iw> w$:
\begin{equation}\label{recursion}
E(X^P_{s_{i}\om})_{[v]}=\delta\!\left(\frac{z_{i+1}}{z_i},
\frac{h^{r_{\om^{-1}(i+1)}}\,y_{\om^{-1}(i+1)}}{h^{r_{\om^{-1}(i)}}\,y_{\om^{-1}(i)}}\right) 
\cdot E(X^P_\om)_{[v]} + 
\delta\!\left(\frac{z_{i}}{z_{i+1}},h\right)\cdot s_i^z
[E(X^P_w)_{[s_iv]}],\end{equation}
where $s_i^z[ f(\ldots,z_i,z_{i+1},\ldots)]=f(\ldots,z_{i+1},z_i,\ldots)$, and the initial condition takes the form
\begin{equation}\label{eqn:Einitial}
E(X^P_{1},\lambda)_v=
\begin{cases}
1 & \text{if }  v=1 \\
0 & \text{if } v\not= 1.
\end{cases}
\end{equation}

\section{Weight functions of \cite{RTV} represent elliptic characteristic classes}\label{sec:weight}

We first introduce a class of special functions called {\it elliptic weight functions}. Then, we show that after a certain shift of variables they represent elliptic classes $E(X^P_\om,\lambda)$ in type $A_{n-1}$.

\subsection{Elliptic weight functions}

As in Subsection \ref{sec:GLn_notations}, we have ${\bf k}= (k_1,\ldots,k_m)\in \Z^m_{\geq 1}$,  $k^{(s)}=\sum_{i=1}^s k_i$ and $n=k^{(m)}=\sum_{i=1}^m  k_i$. The corresponding partial flag variety $\SL_n/P$ parametrizes flags of  subspaces $\{V_i\}_{i=1,\ldots,m}$ with $\dim V_i=k^{(i)}$. The set of cosets $W/W_P$ (in fact,  the set $W^P$) is in natural bijection with the set of partitions $I=(I_1,\ldots,I_m)$ of $\{1,\ldots,n\}$ with $|I_i|=k_i$. For such an $I$ we  use the notation $\cup_{i=1}^s I_i=\{i^{(s)}_1<\ldots<i^{(s)}_{k^{(s)}}\}$. 

Consider the set of variables $t^{(s)}_i$ for $s=1,\ldots, m$, $i=1,\ldots,k^{(s)}$, and set $t^{(m)}_i=z_i$. 
Following \cite{RTV} define the {\em elliptic weight function} by
\[
\W_I=
\frac{1}{\prod_{s=1}^{m-1} \prod_{i=1}^{k^{(s)}} \prod_{j=1}^{k^{(s)}} \vt(ht^{(s)}_j/t^{(s)}_i)} \cdot
\Sym_{t^{(1)}} \ldots \Sym_{t^{(m-1)}} (U_I),
\]
where $\Sym_{t^{(s)}}$ is symmetrization with respect to the  $t^{(s)}_1,\ldots,t^{(s)}_{k^{(s)}}$ variables, and 
\[
U_I=\prod_{s=1}^{m-1} \prod_{a=1}^{k^{(s)}}
       \left(
          \prod_{c=1}^{k^{(s+1)}} \psi_{I,s,a,c}(t^{(s+1)}_c/t^{(s)}_a) 
          \prod_{b=a+1}^{k^{(s)}} \frac{ \vt(ht^{(s)}_b/t^{(s)}_a)}{ \vt(t^{(s)}_b/t^{(s)}_a)}
       \right),
\]
where
\[
\psi_{I,s,a,c}(x)=\vt(x)\cdot 
     \begin{cases}
		\delta(x,h) & \text{if }\  i^{(s+1)}_c<i^{(s)}_a \\
		\delta(x,h^{1+p_{I,j(I,s,a)}(i^{(s)}_a)-p_{I,s+1}(i^{(s)}_a)}\mu_a/\mu_b) & \text{if }\  i^{(s+1)}_c=i^{(s)}_a\\
		1 & \text{if }\  i^{(s+1)}_c>i^{(s)}_a. \\
	\end{cases}
\]
Here $\mu_a:= h^{-\lambda_a}$ and we used the numerical functions 
\begin{itemize}
\item $j(I,s,a)$ is defined by $i^{(s)}_a\in I_{j(I,s,a)}$;
\item $p_{I,j}(i)=| I_j \cap \{1,\ldots,i-1\}|$.
\end{itemize}

For example, for ${\bf k}=(1,1)$ we have (temporarily denoting $t=t^{(1)}_1$)
\begin{equation}\label{n2}
\W_{\{1\},\{2\}}=\frac{ \vt(z_1t^{-1}h\mu_2\mu_1^{-1}) \vt(z_2t^{-1})}{\vt(h\mu_2\mu_1^{-1})},
\qquad
\W_{\{2\},\{1\}}=\frac{ \vt(z_1t^{-1}h) \vt(z_2t^{-1}\mu_2\mu_1^{-1})}{\vt(h)\vt(\mu_2\mu_1^{-1})}.
\end{equation}

More important than the actual formula above---we admit that it is terribly complicated at the first sight---is the recursion for the weight functions phrased in the next two propositions. Recall that the $s_i^z$ operator switches the variables $z_i$ and $z_{i+1}$. The operator $s_i$ acts on a partition $I$ by replacing the numbers $i$ and $i+1$.
 
\begin{proposition}[R-matrix recursion for weight functions]
Assume that for $i\in I_a$, $i+1\in I_b$ we have $a<b$. Then
\begin{equation}\label{RrecursionW}
\W_{s_{i}(I)}= 
   \delta\!\left(\frac{z_{i+1}}{z_{i}},\frac{\mu_b h^{p_{I,a}(i)}}{\mu_a h^{p_{I,b}(i+1)}}\right) \cdot \W_I 
   +
   \delta\!\left(\frac{z_i}{z_{i+1}},h\right) \cdot s^z_{i}[\W_I].
\end{equation}
\end{proposition}

\begin{proof}
This is, in fact, not a new result. The weight functions defined in \cite[Section 2.4]{RTV} only differ from ours by some irrelevant power of $\vt(h)$, and some global factors. For the weight functions defined there, an R-matrix recursion is proved there in Theorem 2.2(2.18). Applying that formula for $\sigma= 1$, renaming $I$ to $s_{i}(I)$, and rearranging, we arrive at \eqref{RrecursionW}.
\end{proof}

For a function $f$ in the variables $t^{(s)}_i$ (e.g. a weight function), and $I\in W/W_P$ let $f|_{I}$ be the function obtained from $f$ by substituting
\[
t^{(s)}_j \ \mapsto \  z_{i^{(s)}_j} \qquad\  \text{for } s=1,\ldots, m-1, j=1,\ldots, k^{(s)}.
\]
Let $I^0$ be the ``smallest'' $I$, that is $I^0_1=\{1,2,\ldots,k_1\}$, $I^0_2=\{k_1+1,\ldots,k_1+k_2\}$, etc. 
\begin{proposition}\label{prop:Winitial}
We have
\[
{\W_{I^0}|}_{J}=
\begin{cases}
\prod_{1\leq a<b\leq m} \prod_{i\in I^0_a} \prod_{j\in I^0_b} \vt(z_j/z_i) & \text{if } J=I^0 \\
0 & \text{if } J\not= I^0.
\end{cases}
\]
\end{proposition}

\begin{proof} 
The statement follows from \cite[Lemmas 2.4, 2.5]{RTV}, or by careful inspection of the formula for the weight function. (The reader is advised to verify the statement by substituting $t=z_1$ or $t=z_2$ in the formula $W_{\{1\},\{2\}}$ in \eqref{n2}, the general case only differs by tracing indexes).
\end{proof}

\subsection{Weight functions versus elliptic classes}
The variables of the weight function $\W_I$ are $t^{(s)}_i$, $z_i$, $\mu_i$, $h$. 
The elliptic class $E(X^P_\om,\lambda)$ lives in the $\T=(\C^*)^n$ equivariant $K$-theory of $G/P$ extended by variables $h$ and $y_j$ (see \eqref{defy}). 

Recall that the partial flag variety $G/P$ parametrizes nested subspaces $V_s$ of dimension $k^{(s)}$. Let the tautological bundle over $G/P$ whose fiber is $V_s$ be denoted by $\TTT_s$. Then, $\TTT_s$ represents an element in $K_{\T}(G/P)$.


Consider the following evaluation of the variables of $\W_I$:
\begin{equation}\label{eq:subs}
\begin{tabular}{lcl}
$t^{(s)}_i$ & $\mapsto$ & \gr roots of $\TTT_s$  \\
$z_i$ & $\mapsto$ & \gr roots of the tautological $n$-bundle over the classifying space $B\T$ \\
$h$ & $\mapsto$ & $h$ \\
$\mu_s$ & $\mapsto$ & $y_j \cdot h^{s-k^{(s-1)}}$ where $j\in [k^{(s-1)}+1,k^{(s)}]$.
\end{tabular}
\end{equation}
Note that the last substitution makes sense, since if $j,j'\in [k^{(s-1)}+1,k^{(s)}]$ then $y_j(\lambda)=y_{j'}(\lambda)$ for $\lambda\in (\t^*)^{W_P}$. 

\begin{theorem} For any $I\in W^P$, the evaluation \eqref{eq:subs} of $\W_I/e^E(T(G/P))$ represents $E(X^P_I,\lambda)$.
\end{theorem}

In other words the evaluation of the weight function $\W_I$ is the $\Ellt^E$-class of the pair \eqref{eqn:pair}.

\begin{proof}
Introducing the notation $e_I=e^E(T(G/P))|_{I}$, from the known description of the tangent space of partial flag varieties we obtain
\[
e_I=\prod_{1\leq a<b\leq n} \,\prod_{i\in I_a} \prod_{j\in I_b} \vt(z_j/z_i).
\]
With this notation we need to show that 
\begin{equation}\label{toprove}
\frac{\W_I|_{_J} }{ e_J} = E(X_I^P,\lambda)_J
\end{equation}
for all $I$ and $J$, which we will  prove by induction on the length of $I$. 
For $I^0$ \eqref{toprove} follows from the comparison of \eqref{eqn:Einitial} and Proposition \ref{prop:Winitial}.

Now, assume that for $i\in I_a$, $i+1\in I_b$ we have $a<b$.  Then, from \eqref{RrecursionW}, for all $J$ we obtain
\begin{equation*}\label{RrecursionWJ}
\W_{s_{i}(I)}|_{J}= 
   \delta\!\left(\frac{z_{i+1}}{z_{i}},\frac{\mu_b h^{p_{I,a}(i)}}{\mu_a h^{p_{I,b}(i+1)}}\right) \cdot \W_I|_J
   +
   \delta\!\left(\frac{z_i}{z_{i+1}},h\right) \cdot \left( s^z_{i}[\W_I] \right)|_J.
\end{equation*}
Using $\left( s^z_{i}[\W_I] \right)|_J=s_i^z[ {\W_I}|_{s_i(J)}]$, and temporarily denoting the left hand side of \eqref{toprove} by $E'(X^P_I)_J$, we can write
\begin{multline*}
E'(X^P_{s_i(I)})_J \cdot e_J=\\
   \delta\!\left(\frac{z_{i+1}}{z_{i}},\frac{\mu_b h^{p_{I,a}(i)}}{\mu_a h^{p_{I,b}(i+1)}}\right) \cdot E'(X^P_I)_J \cdot e_J
   +
   \delta\!\left(\frac{z_i}{z_{i+1}},h\right) \cdot  s_i^z[  E'(X^P_I)_{s_i(J)} ] s_i^z[e_{s_i(J)}].
\end{multline*}
Remarkably, from the explicit formula for $e_I$ we can see that $s_i^z[e_{s_i(J)})]=e_J$. Hence, after division by $e_J$, we arrive at
\begin{equation}\label{almost}
E'(X^P_{s_i(I)})_J =\\
   \delta\!\left(\frac{z_{i+1}}{z_{i}},\frac{\mu_b h^{p_{I,a}(i)}}{\mu_a h^{p_{I,b}(i+1)}}\right) \cdot E'(X^P_I)_J 
   +
   \delta\!\left(\frac{z_i}{z_{i+1}},h\right) \cdot  s_i^z[  E'(X^P_I)_{s_i(J)} ].
\end{equation}
We claim that this recursion is the same as the recursion for $E(X^P_I)_J$ given in \eqref{recursion}, which will complete our proof. Hence, we only need to identify the coefficient of $E(X_I)_J$ in \eqref{recursion} with the coefficient of $E'(X_I)_J$ in \eqref{almost}---after the substitution \eqref{eq:subs}. That is, we need the combinatorial statement 
\[
r_{\om^{-1}(i+1)}-r_{\om^{-1}(i)}=(b-k^{(b-1)}+p_{I,a}(i))-(a-k^{(a-1)}+p_{I,b}(i+1)),
\]
or equivalently, that the quantity
$$p_{I,a}(i)+
r_{\om^{-1}(i)}+
k^{(a-1)}-a
$$
does not depend on $i$ ($a$ is determined by $i$ via $i\in I_a$).
Tracing back the definitions of these combinatorial functions we see that 
\begin{itemize}
\item $p_{I, a}(i)+k^{(a-1)}+1=w^{-1}(i)$, and
\item $a-w^{-1}(i)$ works for a choice of $r_{w^{-1}(i)}$ (recall from Section \ref{sec:GLn_notations} that $r_j$'s are only defined up to a uniform scalar addition).
\end{itemize}
From these two claims, by cancelling $w^{-1}(i)$, we obtain that $p_{I, a}(i) + k^{(a-1)} + r_{w^{-1}(i)} - a = -1$, that is, a number independent of $i$. This completes the proof. 
\end{proof}

\begin{example} \rm
Let $k=(2,3,2)$, and choose $\om=[i_1,i_2,\ldots,i_7]\in W^P$. For the corresponding $I=(\{i_1,i_2\},\{i_3,i_4,i_5\},\{i_6,i_7\})$ the various combinatorial functions 

\centerline{
\begin{tabular}{c|cc|ccc|cc}
$i$ & $i_1$ & $i_2$ & $i_3$ & $i_4$ & $i_5$ & $i_6$ & $i_7$ \\ \hline
$w^{-1}(i)$ & 1 & 2 & 3 & 4 & 5 & 6 & 7 \\
$a$ & 1 & 1 & 2 & 2 & 2 & 3 & 3\\
$r_{w^{-1}(i)}$ & 0 & -1 & -1 & -2 & -3 & -3 & -4 \\
$k^{(a-1)}$ & 0 & 0 & 2 & 2& 2 & 5 & 5\\
$p_{I, a}(i)$ & 0 & 1& 0 & 1 & 2 & 0 & 1\\
\end{tabular}
}

\noindent illustrate the last, combinatorial, part of the proof above, namely the identity $p_{I, a}(i) + k^{(a-1)} + r_{w^{-1}(i)} - a = -1$. 
\end{example}

\section{Remarks}

\subsection{Transformation properties}\label{sec:trans}
Equivariant elliptic cohomology classes of a point can be regarded as sections of certain line bundles over some products of elliptic curves. Hence, the function $E(X_w^P)_v$ can be regarded as a section of a  line bundle (depending on $G, P, w, v$) over a product of elliptic curves. For example, the function \eqref{eqn:ESp} can be regarded as a section of a line bundle over $E^4$, where $E=\C^*/(q^{\Z})$, $q=e^{2\pi i \tau}$, and the coordinates of the factors of $E^4$ are $z_1, z_2, h, \mu$.

By {\em transformation property} we mean the following: To a product of theta functions we associate a quadratic form as follows: to $\vt(\prod_{i=1}^p x_i^{r_i} )$ associate $(\sum_{i=1}^p r_ix_i)^2$, and to a product of $\vt$-functions associate the sum of the quadratic forms of each factor. For a more conceptual explanation see \cite[Section 5]{FRV18}. 
For example, the quadratic forms associated to the three terms of the function \eqref{eqn:ESp} are (up to the same scalar multiple)
\[-2 z_2( \mu-h) +  (z_2 - z_1) ( 2 \mu-h)+ (-2  z_2)\mu,\]
\[(-z_1 - z_2)( 2 \mu-h)\qquad\text{and}\qquad     - 2 z_2( \mu-h)+  (z_1 - z_2) h+(-2 z_1)\mu.\]
The reader can trivially  verify that these three quadratic forms are all equal. 

The general fact that the different summands of $E(X_w^P)_v$ must have the same {transformation property} (i.e., the same associated quadratic form) is a useful practical reality check in calculations.  

\subsection{Axiomatic characterization}
The fact that \charr classes of Schubert (or other geometrically relevant) varieties can be described by axioms turned out useful in several parts of enumerative geometry. Such axiomatic characterizations were initially known for the cohomological fundamental class \cite{Rinv}, but,  after Okounkov's works, such axiomatic characterizations are proved for the cohomological CSM classes and for the $K$-theoretic MC classes as well \cite{RV, FR, FRW1, FRW2}. 

It can be shown that the elliptic classes of Schubert varieties studied in this paper have an axiomatic characterization, too. However, no argument in this paper relies on such characterization, and, in fact, even phrasing the axioms precisely would be rather technical. Hence, here we only {\em sketch} the axiomatic characterization briefly.

\smallskip

\noindent The $E(X_w^P)_v$ functions satisfy:
\begin{enumerate}
\item \label{itemGKM} (GKM axiom) Let $\alpha:\T\to \C^*$ be a root of $G$ (not necessarily simple). If $v_1=v_2 s_\alpha$ then $$\big(E(X_w^P)_{v_1}-E(X^P_w)_{v_2}\big){}_{|\ker(\alpha)\times (\t^*)^{W_P}}=0\,.$$ 
Here the restriction of the elliptic class is considered as a function on $\t\times (\t^*)^{W_P}$.
\item \label{support} (support axiom) In the appropriate sense, the class $E(X_w^P)$ is supported on the union of the conormal spaces of Schubert cells $X^P_v$ for $v\leq w$. To make sense of this condition, first one needs to interpret $E(X_w^P)$ as an element of the $K$-theory or elliptic cohomology of the cotangent bundle of $G/P$ (using $h$ as the first Chern class of an extra $\C^*$ action scaling the fibers). Then, the support condition means that the class $E(X_w^P)$ restricted to the complement of the named union is 0. A more practical interpretation (which can be phrased without involving the cotangent bundle) is that the local classes $E(X_w^P)_v$ satisfy certain divisibility properties. (For an argument reducing the support condition to a set of divisibility conditions see \cite[Proof of Thm 5.1]{RTV}.) 
\item \label{normalization} (normalization axiom)  The `diagonal' local classes are $E(X^P_\om)_\om=\prod\delta(\chi,h)$ for the weights $\chi$ of $T_\om (X^P_\om)$.
\end{enumerate}
The axiomatic characterization theorem for the $E(X^P_w)$ classes states that if a collection of functions $f_{w,v}$ satisfy the three listed conditions, as well as the transformation property of $f_{w,v}$ are the same as those of $E(X^P_w)_v$, then $f_{w,v}=E(X^P_w)_v$. For analogous arguments see \cite[3.3.5]{AO}, \cite[Sec.7.8]{RTV}.

\end{document}